\documentclass[11pt]{article}

\usepackage[T1]{fontenc}
\usepackage{amsmath,amssymb,amsthm,thmtools,mathtools}
\usepackage{enumitem}
\usepackage{etoolbox}
\usepackage{tikz}
\usepackage{graphicx,caption,subcaption}
\usepackage[table]{xcolor}
\usepackage[colorlinks=true,linkcolor=blue!55!black,citecolor=blue!55!black,urlcolor=blue!55!black]{hyperref}
\usepackage[capitalise,noabbrev]{cleveref}
\usepackage{kotex}

\usepackage{latexsym}
\usepackage{cite}
\usepackage[normalem]{ulem}

\AtBeginEnvironment{thebibliography}{\setlength{\itemsep}{2pt}}

\newtheorem{theorem}{Theorem}[section]
\newtheorem{lemma}[theorem]{Lemma}
\newtheorem{proposition}[theorem]{Proposition}
\newtheorem{corollary}[theorem]{Corollary}

\newtheorem{conjecture}[theorem]{Conjecture}

\newtheorem{problem}[theorem]{Problem}
\theoremstyle{definition}

\theoremstyle{remark}

\DeclareMathOperator{\mad}{mad}
\newcommand{\socol}{\chi_{\mathrm{so}}}
\newcommand{\ocol}{\chi_{\mathrm{o}}}
\newcommand{\cset}{\mathcal{C}}
\newcommand{\eps}{\varepsilon}
\newcommand{\vp}{\varphi}

\title{On strong odd colorings of sparse graphs}

\author{
Ilkyoo Choi\thanks{Department of Mathematics, Hankuk University of Foreign Studies, Yongin-si, Gyeonggi-do, Republic of Korea.
 E-mail: \texttt {ilkyoo@hufs.ac.kr},
  Discrete Mathematics Group, Institute for Basic Science (IBS), Daejeon, Republic of Korea, and Center for AI and Natural Sciences, Korea Institute for Advanced Study (KIAS), Seoul, Republic of Korea.
Supported by the Hankuk University of Foreign Studies Research Fund, the National Research Foundation of Korea (NRF) grant funded by the Korea government (MSIT) (RS-2025-23324220), Institute for Basic Science (IBS-R029-C1), and the Korea Institute for Advanced Study (KIAS) grant funded by the Korea government.
}
\and
Hojin Chu\thanks{School of Computational Sciences, Korea Institute for Advanced Study (KIAS), Seoul, Republic of Korea. E-mail: \texttt{hojinchu@kias.re.kr}. Supported by a KIAS Individual Grant (CG101801) at Korea Institute for Advanced Study.}
}

\date{}

\begin{document}

\maketitle

\begin{abstract}
A \emph{strong odd coloring} of a graph is a proper vertex coloring such that for every vertex $v$, each color appearing on the neighborhood of $v$ appears an odd number of times.
We investigate strong odd colorings of sparse graphs, where sparsity of a graph $G$ is measured in terms of maximum average degree, denoted $\mad(G)$. 

We prove that the optimal threshold on the maximum average degree guaranteeing a strong odd $c$-coloring is $4-\Theta(\frac1c)$ as $c\to\infty$. 
Moreover, for every $c\geq 5$, this threshold is strictly smaller than the threshold for odd $c$-coloring. 
We also determine the sharp threshold on the maximum average degree guaranteeing a strong odd $5$-coloring: every graph $G$ with $\mad(G)<\frac73$ has a strong odd $5$-coloring, and this bound is sharp. 
In addition, we show that every graph $G$ with $\mad(G)<\frac{11}5$ and no induced $5$-cycle has a strong odd $4$-coloring.
Note that the $5$-cycle has maximum average degree 2 and no strong odd $4$-coloring. 
\end{abstract}

\section{Introduction}

All graphs in this paper are finite and simple, unless stated otherwise.
For a graph $G$, let $V(G)$, $E(G)$, and $\Delta(G)$ denote its vertex set, edge set, and maximum degree, respectively. 
The complete graph on $n$ vertices is denoted by $K_n$.
A \emph{proper $c$-coloring} of a graph $G$ is a function $\vp:V(G)\to\{1, \ldots, c\}$ that assigns distinct colors to adjacent vertices, and the \emph{chromatic number} of $G$, denoted $\chi(G)$, is the minimum $c$ for which $G$ has a proper $c$-coloring. 

Recently introduced by Kwon and Park~\cite{KwonPark2026}, a \emph{strong odd $c$-coloring} of a graph $G$ is a proper $c$-coloring such that for each vertex $v$, every color appearing on its neighborhood $N_G(v)$ appears an odd number of times.
The \emph{strong odd chromatic number} of $G$, denoted $\socol(G)$, is the minimum $c$ for which $G$ has a strong odd $c$-coloring.
Note that every graph has a strong odd coloring, as we can give every vertex its own color. 

Strong odd coloring lies between two well-studied coloring notions.
The first is an \emph{odd $c$-coloring}, introduced by Petru\v{s}evski and \v{S}krekovski~\cite{PetruSevskiSkrekovski2022}, which is a proper $c$-coloring such that, for every non-isolated vertex $v$, there exists a color appearing an odd number of times on $N_G(v)$.
The second is a proper coloring of the \emph{square} of $G$, denoted $G^2$, which is the graph obtained from $G$ by adding edges between vertices at distance $2$.
For a graph $G$, every proper coloring of $G^2$ is a strong odd coloring of $G$, since for every vertex $v$, the vertices in $N_G(v)$ receive pairwise distinct colors. 
Therefore, $\chi(G)\le \ocol(G)\le \socol(G)\le \chi(G^2)$, where $\ocol(G)$ is the minimum $c$ for which $G$ has an odd $c$-coloring. 

Even for a sparse graph $G$, strong odd coloring of $G$ and proper coloring of $G^2$ behave quite differently. 
Indeed, the closed neighborhood $N_G(v)\cup\{v\}$ of a vertex in $G$ becomes a clique in $G^2$, so $\chi(G^2)\geq\Delta(G)+1$. 
In particular, squares of planar graphs have unbounded chromatic number.
In contrast, there exists a minimum constant $C_{\mathrm{so}}$ for which all planar graphs have a strong odd $C_{\mathrm{so}}$-coloring. 
Caro et al.~\cite{CaroPetrusevskiSkrekovskiTuza2025} proved that $C_{\mathrm{so}}\leq 388$, which was recently improved to $C_{\mathrm{so}}\leq 368$ by 
Santra~\cite{Santra2026}.
On the other hand, $C_{\mathrm{so}}\geq 20$, as there is a planar graph with strong odd chromatic number $20$~\cite{GoetzeEtAl2025,Santra2026}.
More generally, Goetze et al.~\cite{GoetzeEtAl2025} showed  that every proper minor-closed class has bounded strong odd chromatic number, and Pilipczuk~\cite{Pilipczuk2025} extended this result to classes of bounded expansion.
Nevertheless, the strong odd chromatic number cannot be bounded by a function of the odd chromatic number~\cite{GoetzeEtAl2025}.

In this paper, we study strong odd colorings of sparse graphs, where sparsity is measured in terms of \emph{maximum average degree}; the maximum average degree of a graph $G$, denoted $\mad(G)$, is $\displaystyle\max_{\emptyset\ne F\subseteq G}\frac{2|E(F)|}{|V(F)|}$.
Kwon and Park~\cite{KwonPark2026} proved that every graph $G$ with $\mad(G)\le\frac{20}7$ satisfies   $\socol(G)\le\Delta(G)+4$, and every graph $G$ with no $4$-cycle and $\mad(G)\le \frac{30}{11}$ satisfies $\socol(G)\le\Delta(G)+3$.
In contrast, we seek conditions on the maximum average degree that guarantee a strong odd $c$-coloring for a given number $c$ of colors.

The corresponding problem for odd coloring is completely resolved for $c\geq 5$: the exact threshold on the maximum average degree guaranteeing an odd $c$-coloring is known~\cite{Cranston2024, ChoChoiKwonPark2023}. 
Namely, for $c\geq 5$, if $G$ is a graph with $\mad(G)<\frac{4c}{c+2}$, then $\ocol(G)\le c$.
The bound on the density is sharp: the graph $S(K_{c+1})$ obtained by subdividing every edge of $K_{c+1}$ exactly once satisfies 
$\mad(S(K_{c+1}))=\frac{4c}{c+2}$ and $\ocol(S(K_{c+1}))=c+1$.
Therefore, for every $\eps>0$, the odd chromatic number is bounded on the class of graphs $G$ with $\mad(G)<4-\eps$, whereas it is unbounded on the class of graphs $G$ with $\mad(G)<4$.

Our first main result establishes the analogous phenomenon for strong odd coloring.
Namely, for every $\eps>0$, the strong odd chromatic number is bounded on the class of graphs $G$ with $\mad(G)<4-\eps$.
We also determine the optimal order of this bound as $\eps\to0$.

\begin{theorem}\label{thm:c-bound}
Let $c\ge54$ be an integer.
For a graph $G$, if $\mad(G)<4-\frac{96}{c-65+\sqrt{(c-65)^2+1152}}$,
then $\socol(G)\le c$.
\end{theorem}

Note that the density bound in \Cref{thm:c-bound} is $4-\frac{48}c+O(\frac1{c^2})$.
We also construct a family of graphs $H_c$ satisfying $\mad(H_c)=4-\frac{8}{c}+O(\frac1{c^2})$ and 
$\socol(H_c)>c$.
Therefore, the density threshold guaranteeing a strong odd $c$-coloring is $4-\Theta(\frac1c)$ as $c\to\infty$.
Equivalently, the minimum number of colors required for a strong odd coloring of every graph $G$ with $\mad(G)<4-\eps$ is $\Theta(\frac1\eps)$ as $\eps\to0$. 
Moreover, $H_c$ satisfies $\mad(H_c)<\frac{4c}{c+2}$ for every $c\geq 5$. 
Thus, $\ocol(H_c)\leq c<\socol(H_c)$, so the threshold on the maximum average degree guaranteeing a strong odd $c$-coloring is strictly smaller than that for odd $c$-coloring.

We next consider strong odd colorings with few colors.
For five colors, we obtain the following sharp result.

\begin{theorem}\label{thm:main-five}
For a graph $G$, if $\mad(G)<\frac73$, then $\socol(G)\le5$.
\end{theorem}

The bound on the maximum average degree in \Cref{thm:main-five} is sharp as there is a graph $Q_5$ with $\mad(Q_5)=\frac73$ and no strong odd $5$-coloring. 

For four colors, the $5$-cycle has maximum average degree $2$ and no strong odd $4$-coloring. 
Nonetheless, excluding induced $5$-cycles allows us to obtain a bound greater than 2 on the maximum average degree that guarantees a strong odd $4$-coloring.

\begin{theorem}\label{thm:main-four}
If a graph $G$ has no induced $5$-cycles and $\mad(G)<\frac{11}5$, then $\socol(G)\le4$.
\end{theorem}

The graph consisting of two $4$-cycles sharing exactly one vertex has no induced $5$-cycle,  has maximum average degree $\frac{16}7$, and has no strong odd $4$-coloring.

Our results imply corollaries regarding planar graphs with girth restrictions. 
Every planar graph with girth at least $g\geq3$ has maximum average degree less than $\frac{2g}{g-2}$.
From our results, we obtain the following corollary:

\begin{corollary}\label{cor:four-planar}
Let $G$ be a planar graph with girth at least $g$.
\begin{itemize}
    \item If $g\geq 5$, then $\socol(G)\le6\left\lceil
12+\frac{4}{g-2}+\frac{8}{g-4}
\right\rceil$.
    \item If $g\geq 14$, then $\socol(G)\leq 5$.
    \item If $g\geq 22$, then $\socol(G)\leq 4$.
\end{itemize}
\end{corollary}

The paper is organized as follows.
In \Cref{sec:examples}, we present the graph $Q_5$ showing that \Cref{thm:main-five} is sharp and the graph $H_c$ related to \Cref{thm:c-bound}.
In \Cref{sec:preliminaries}, we prove structural lemmas to establish the reducible configurations needed to prove  \Cref{thm:main-five,thm:main-four}, which are completed in \Cref{sec:five,sec:four}, respectively, via discharging. 
In \Cref{sec:below-four}, we prove \Cref{thm:c-bound} using a theorem of Lu\v{z}ar, Petru\v{s}evski, and \v{S}krekovski~\cite{LuzarPetrusevskiSkrekovski2015} on odd edge-colorings of loopless multigraphs.
 % admits an odd $6$-edge-coloring.
Finally, we suggest a potential future research direction in \Cref{sec:concluding}.

We end this section with some notation used throughout the paper. 
For a graph $G$ and a vertex $v$, let $d_G(v)$, $N_G(v)$, and $N_G[v]$ denote the degree, open neighborhood, and closed neighborhood of $v$, respectively.
Subscripts are omitted when the graph is clear.
A \emph{$k$-vertex}, a \emph{$k^+$-vertex}, and a \emph{$k^-$-vertex} are vertices of degree $k$, at least $k$, and at most $k$, respectively.
A \emph{$k$-neighbor}, a \emph{$k^+$-neighbor}, and a \emph{$k^-$-neighbor} are neighbors of degree $k$, at least $k$, and at most $k$, respectively.

\section{Extremal examples}\label{sec:examples}

In this section, we construct sparse graphs that have no strong odd $c$-coloring. 
Note that a subgraph $F$ of $G$ with average degree $\mad(G)$ is always an induced subgraph.
Moreover, if $\mad(G)>2$, then $F$ must have minimum degree at least $2$, otherwise we can remove a vertex of degree at most $1$ from $F$ to increase its average degree. 

\subsection{Five colors}

Let $Q_5$ be the graph obtained by connecting two vertices $u,v$ with three internally vertex-disjoint paths $upv$, $uqv$, $uxyv$. 
We first show $\mad(Q_5)=\frac73$. 
Note that the average degree of $Q_5$ is $\frac73$. 
Let $H$ be a (induced) subgraph of $Q_5$ with the maximum average degree; note that $H$ must have minimum degree at least $2$.  
If $H$ is a cycle, then $\mad(H)=2$.
Otherwise, $H$ is the entire graph $Q_5$, so  $\mad(H)=\frac73$.
Hence $\mad(Q_5)=\frac73$.

Now, we prove $Q_5$ has no strong odd $5$-coloring. 
Suppose to the contrary that $Q_5$ has a strong odd $5$-coloring $\vp$.
Every two distinct vertices in $\{u,x,y,v\}$ are either adjacent or the neighborhood of a $2$-vertex of $Q_5$. 
Therefore, all four vertices in $\{u, x, y, v\}$ must receive distinct colors. 
Since $p$ is adjacent to $u$ and $v$, $\vp(p)\not\in\{\vp(u), \vp(v)\}$. 
If $\vp(p)=\vp(x)$, then, since $N(u)=\{p,q,x\}$, $\vp(q)$ must be the same as $\vp(x)$, so $\vp(x)$ appears exactly twice on $N(v)$, which is a contradiction. 
Thus $\vp(p) \neq \vp(x)$.
Similarly, $\vp(p) \neq \vp(y)$.
Hence, $\vp(p)$ does not appear on $\{u, v, x, y\}$.
By symmetry, $\vp(q)$ also does not appear on $\{u, v, x, y\}$.
Then $\vp(p)=\vp(q)$, which is a contradiction since now $\vp(p)$ appears exactly twice on $N(u)$.
Therefore $Q_5$ has no strong odd 5-coloring. 

\subsection{A general construction}

Given a graph $F$, let $S(F)$ denote the graph obtained from $F$ by subdividing each edge exactly once; the \emph{branch} vertices are the vertices that existed before the subdivision. 
For an integer $c\ge4$, set 
$a=\left\lfloor\frac c2\right\rfloor,
b=\left\lceil\frac c2\right\rceil$, and 
$T_c=a^2+b^2=\left\lceil\frac{c^2}{2}\right\rceil$.
We define the graph $H_c$ from the disjoint union of $S(K_a)$ and $S(K_b)$ by adding a vertex $z$ adjacent to all $c$ branch vertices. % (see \Cref{fig:Hc}).
Note that $H_c$ has $a+{a\choose 2}+b+{b\choose 2}+1=\frac{a^2+a+b^2+b+2}{2}=\frac{T_c+c+2}{2}$ vertices and $2{a\choose 2}+2{b\choose 2}+a+b=a^2+b^2=T_c$ edges.

We first show $H_c$ has no strong odd $c$-coloring.
Suppose to the contrary that $H_c$ has a strong odd $c$-coloring $\vp$.
All neighbors of $z$ are branch vertices from $S(K_a)$ and $S(K_b)$. 
Two branch vertices $u, v$ from the same copy of $S(K_r)$ for some $r\in\{a,b\}$ form the neighborhood of a $2$-vertex, so $u, v$ must receive distinct colors.
Therefore if two branch vertices receive the same color, then the color appears exactly twice on $N(z)$, which is a contradiction. 
Thus all $c$ vertices in $N(z)$ receive pairwise distinct colors. 
Now, there is no color for $z$, which is a contradiction. 
Hence, $H_c$ has no strong odd $c$-coloring.

We show that
$\mad(H_c)=\frac{4T_c}{T_c+c+2}=4-\frac{8(c+2)}{c^2+2c+4+\eps}$ where $\eps\in\{0,1\}$ and $\eps \equiv c \pmod 2$. 
Note that the graph $H_c$ has average degree $\frac{4T_c}{T_c+c+2}>2$.
Let $F$ be a (induced) subgraph of $H_c$ with maximum average degree; note that $F$ must have minimum degree at least $2$. 
Suppose $F$ has $p$ and $q$ branch vertices from $S(K_a)$ and $S(K_b)$, respectively.
Since $F$ has minimum degree at least $2$, we have $p+q\ge2$.
The graph $H_c$ is $2$-degenerate, so every subgraph of $H_c$ has average degree less than $4$.
Thus, adding $z$ to $F$ does not decrease its average degree, so we may assume that $F$ contains $z$.
Moreover, whenever $F$ contains two branch vertices $u, v$ from the same copy of $S(K_r)$ for some $r\in\{a,b\}$, adding the $u,v$-path of length $2$ to $F$ does not decrease its average degree.
Therefore we may assume that $F$ has $p^2+q^2$ edges and $\frac{p^2+q^2+p+q+2}{2}$ vertices. 
Hence the average degree of $F$ is $\mad(H_c)=g(p,q)=\frac{4(p^2+q^2)}{p^2+q^2+p+q+2}$ where $0\le p\le a$ and $0\le q\le b$.
Maximizing $g(p,q)$ is equivalent to maximizing $f(p,q)=\frac{p^2+q^2}{p+q+2}$, since $g(p, q)=\frac{4f(p,q)}{f(p,q)+1}$ and the function $4x/(x+1)$ is strictly increasing for $x\geq 0$. 
For fixed $q\ge0$, the function $f(p,q)$ is strictly convex in $p$, since $\frac{\partial^2 f}{\partial p^2}=\frac{4(q^2+2q+2)}{(p+q+2)^3}>0$.
By symmetry, $f$ is strictly convex in each variable, so every maximizer in $[0,a]\times[0,b]$ must lie at a corner.
Since $b\ge a$, we have $f(0,0)\le f(a,0)\le f(0,b)$.
Moreover, $f(a,b)-f(0,b)=\frac{a(a(b+2)-b^2)}{(a+b+2)(b+2)}>0$ since $b\in\{a,a+1\}$ and $a\ge2$ imply $a(b+2)-b^2\ge a-1>0$.
Thus $f$ is maximized at a unique point $(a,b)$, so $F=H_c$.

\section{Reducible configurations}\label{sec:preliminaries}

Fix an integer $c\ge3$.
A graph $G$ is \emph{strong odd $c$-minimal} if $G$ has no strong odd $c$-coloring but every proper subgraph of $G$ has a strong odd $c$-coloring.
In this section, we identify several reducible configurations that cannot occur in a strong odd $c$-minimal graph $G_c$.
By minimality, $G_c$ is connected.
We will use the color set $\cset_c$ of $c$ colors; in particular, $\{\alpha,\beta,\gamma, \delta\}\subseteq\cset_c$ when $c\geq 4$. 

An \emph{open $t$-thread} is a path between distinct vertices $x,y$ with exactly $t$ internal vertices, which are all $2$-vertices; $x,y$ are its \emph{terminals}.
A \emph{closed $t$-thread} is a cycle with  exactly one $3^+$-vertex and $t$ $2$-vertices; the $3^+$-vertex is its \emph{terminal}. 
A thread is \emph{maximal} if it is not properly contained in a longer thread.  
We say that a terminal is \emph{incident with} the thread. 
For a positive integer $s$, an \emph{$s^+$-thread} is a $t$-thread where $t \ge s$.

A $2$-vertex $v$ is \emph{nearby} a $3^+$-vertex $u$ if $v$ lies on a thread incident with $u$.
A $3^+$-vertex $u$ is \emph{responsible} for $r$ nearby $2$-vertices if the total number of nearby $2$-vertices of $u$, counted with multiplicity, is $r$, where a nearby $2$-vertex on a maximal open and closed thread incident with $u$ is counted once and twice, respectively. 

\subsection{Common configurations}

\begin{lemma}\label{lem:common-leaves}
If $c\ge3$, then $G_c$ satisfies the following:
\begin{enumerate}[label=\textup{(\roman*)}]
\item Every vertex has at most one $1$-neighbor.
\item If a vertex $u$ has a $1$-neighbor, then $d(u)\ge c$ and $d(u)\equiv c\pmod2$.
In particular, a $1$-vertex has a $c^+$-neighbor, and no $(c+1)$-vertex has a $1$-neighbor.
\end{enumerate}
\end{lemma}
\begin{proof}
(i): Suppose to the contrary that a vertex $u$ has two $1$-neighbors $x$ and $y$.
Let $\vp$ be a strong odd $c$-coloring of $G_c-\{x,y\}$.
If $u$ has no neighbors in $G_c-\{x,y\}$, then $G_c$ is a path on three vertices, which has a strong odd 3-coloring. 
Otherwise, $u$ has a neighbor $z$ distinct from $x,y$, so color both $x$ and $y$ with $\vp(z)$.
In either case, we obtain a strong odd $c$-coloring of $G_c$, which is a contradiction.  

(ii): Suppose that a vertex $u$ has a $1$-neighbor $x$. 
Let $\vp$ be a strong odd $c$-coloring of $G_c-x$. 
If a color $\alpha$ does not appear in $\vp(N[u]\setminus\{x\})$, then let $\vp(x)=\alpha$ to obtain a strong odd $c$-coloring of $G_c$, which is a contradiction.
Otherwise, all $c$ colors appear in $\vp(N[u]\setminus\{x\})$, so $d(u)\geq c$. 
Moreover, every color distinct from $\vp(u)$ appears on $N(u)\setminus\{x\}$ an odd number of times, so 
$d(u)=1+\sum_{\alpha\in\cset_c\setminus\{\vp(u)\}} \bigl|(N(u)\setminus\{x\})\cap\vp^{-1}(\alpha)\bigr|\equiv 1+(c-1) \pmod{2}$.
\end{proof}

\begin{lemma}
\label{lem:common-closed-thread}
Let $c\ge4$. If a vertex $u$ of $G_c$ is incident with a closed $t$-thread, then $t\in\{2,3\}$ and $u$ is a $c^+$-vertex. 
Moreover, if $u$ has a $1$-neighbor, then $u$ is a $(c+2)^+$-vertex.
\end{lemma}
\begin{proof}
Let $T=ux_1x_2\ldots x_tu$ be a closed $t$-thread.
Let $\vp$ be a strong odd $c$-coloring of $G_c-\{x_1,\ldots,x_t\}$.
Since $u$ is a $3^+$-vertex, $u$ has a neighbor $w$ not on $T$.
Let $\alpha=\vp(u)$ and $\beta=\vp(w)$.
Suppose to the contrary that $t\ge4$.
Extend $\vp$ to $x_1, x_t$ by coloring $x_1, x_t$ with $\beta$. 
Now, we can easily extend this coloring to a strong odd $c$-coloring of $G_c$, a contradiction.
Therefore $t\in\{2,3\}$.
Suppose to the contrary that $d(u)\le c-1$.
Since $|N[u]\setminus\{x_1, x_t\}|=d(u)-1\le c-2$, there are distinct colors $\gamma,\delta\notin\vp(N[u]\setminus\{x_1, x_t\})$.
Extend $\vp$ to $x_1, x_t$ by coloring $x_1, x_t$ with $\gamma, \delta$, respectively. 
If $t=2$, then we colored all of $G_c$. 
Otherwise $t=3$, so color the only uncolored vertex $x_2$ with $\beta$. 
In either case, we obtain a strong odd $c$-coloring of $G_c$, which is a contradiction.  

For the moreover part, suppose to the contrary that $u$ is a $(c+1)^-$-vertex with a $1$-neighbor $z$.
Since $u$ is a $c^+$-vertex and $c\ge 4$, $u$ has a neighbor $w\not\in\{z, x_1, x_t\}$.
Let $\psi$ be a strong odd $c$-coloring of $G_c-\{z, x_1, \ldots, x_t\}$. 
Let $\alpha=\psi(u)$ and $\beta=\psi(w)$.
Since $|N[u]\setminus\{z, x_1, x_t\}|=d(u)-2\leq c-1$, there is a color $\gamma\not\in\psi(N[u]\setminus\{z,x_1,x_t\})$.
In particular, $\gamma\notin\{\alpha,\beta\}$.
Extend $\psi$ to $z, x_1, x_t$ by coloring $x_t$ with $\gamma$ and both $z$ and $x_1$ with $\beta$. 
If $t=2$, then we colored all of $G_c$. 
Otherwise $t=3$, so color the only uncolored vertex $x_2$ with a color not in  $\{\alpha,\beta,\gamma\}$. 
In either case, we obtain a strong odd $c$-coloring of $G_c$, which is a contradiction.  
\end{proof}

\subsection{Four colors}

\begin{lemma}\label{lem:short-path-extensions}
Let $ux_1\ldots x_tv$ be an open $t$-thread of $G_4$ with $t\ge2$ and let $H$ be an induced subgraph of $G_4-\{x_1, \ldots, x_{t-1}\}$ where $u, x_t, v\in V(H)$. 
Suppose that $\vp$ is a proper $4$-coloring of $H$ where the strong odd condition is not violated at any vertex of $H$, except possibly at $u$.
If either $t=2$ and $\vp(u)\ne\vp(x_2)$ or $t\ge3$, then $\vp$ can be extended to $x_1,\ldots,x_{t-1}$ so that the resulting coloring is proper and the strong odd condition is not violated at any vertex of $G_4[V(H)\cup\{x_1, \ldots, x_{t-1}\}]$, except possibly at $u$.
Moreover, if either $t=3$ and $\vp(u)=\vp(x_3)$ or $t=4$ and $\vp(u)\ne\vp(v)$, then $x_1$ may be assigned any color other than $\vp(u)$ in such an extension. 
\end{lemma}
\begin{proof}
If $t=2$ and $\vp(u)\neq\vp(x_2)$, then color $x_1$ with a color not in $\vp(\{u, x_2, v\})$. 
Suppose $t\geq 3$.
Sequentially, for each $i$ from $t-1$ to $2$, color $x_i$ with a color not appearing on vertices within distance $2$; $x_{i-1}$ is not colored yet, so there are at most three colors to avoid. 
Finally, color $x_1$ with a color not in $\vp(\{u, x_2, x_3\})$. 

We now prove the moreover part.
Assign a color other than $\vp(u)$ for $x_1$. 
If $t=3$ and $\vp(u)=\vp(x_3)$, then color $x_2$ with a color not in $\vp(\{u,x_1,v\})$. 
Assume $t=4$ and $\vp(u)\neq\vp(v)$. 
If $\vp(x_4)\neq \vp(u)$, then color $x_3$ with $\vp(u)$ and color $x_2$ with a color not in $\vp(\{u,x_1,x_4\})$. 
Otherwise, $\vp(x_4)=\vp(u)$, so color $x_3$ with a color not in $\vp(\{u,x_1,v\})$ and color $x_2$ with a color not in $\vp(\{u,x_1,x_3\})$. 
\end{proof}

\begin{lemma}\label{lem:five-path-extensions}
$G_4$ satisfies the following:
\begin{enumerate}[label=\textup{(\roman*)}]
\item There is no open $6^+$-thread.
\item A $3$-vertex is not incident with two internally disjoint open  $5^+$-threads.
\item A $4$-vertex with a $1$-neighbor is not incident with an open $5^+$-thread. 
\end{enumerate}
\end{lemma}
\begin{proof}
    (i): Suppose to the contrary that $ux_1\ldots x_6v$ is an open $6$-thread  of $G_4$.
    Let $\vp$ be a strong odd $4$-coloring of $G_4-\{x_2,x_3,x_4\}$.
    We may assume that $\vp(u)\neq \vp(x_5)$; otherwise, recolor $x_5$ with a color not in $\vp(\{u,x_6,v\})$.
    Then color $x_4$ with a color not in $\vp(\{x_5,x_6\})$.
    By the moreover part of \Cref{lem:short-path-extensions}, we can extend $\vp$ to a strong odd $4$-coloring of $G_4$, which is a contradiction. 

    (ii): Suppose to the contrary that a $3$-vertex $u$ is incident with two open $5^+$-threads $ux^1_1\ldots x_5^1v^1$ and $ux^2_1\ldots x^2_5v^2$. 
    Let $w$ be the neighbor of $u$ distinct from $x^1_1, x^2_1$.
    Let $\vp$ be a strong odd $4$-coloring of $G_4-\{x^1_1,x^1_2,x^1_3,x^1_4,x^2_1,x^2_2,x^2_3,x^2_4\}$.
    Let $\alpha, \beta$ be distinct colors not in $\{\vp(u),\vp(w)\}$. 
    If $\vp(\{v^1,v^2\})\subseteq\{\alpha,\beta\}$, then extend $\vp$ to $x^1_1, x^2_1$ by coloring both $x^1_1, x^2_1$ with $\vp(w)$. 
    Otherwise, extend $\vp$ to $x^1_1, x^2_1$ by coloring $x^1_1,x^2_1$ with $\alpha,\beta$ or $\beta,\alpha$, respectively, so that $\vp(x^1_1)\neq \vp(v^1)$ and $\vp(x^2_1)\neq \vp(v^2)$.
    In either case, for each $i\in\{1,2\}$, color $x^i_2$ with a color not in $\vp(\{u,x^i_1\})$.
    By the moreover part of \Cref{lem:short-path-extensions}, 
    % applied to $x^1_1\ldots x_5^1v^1$ and $x^2_1\ldots x^2_5v^2$, 
    we can extend $\vp$ to a strong odd $4$-coloring of $G_4$, which is a contradiction.

    (iii): Suppose to the contrary that a $4$-vertex $u$ has a $1$-neighbor $w$ and is incident with an open $5^+$-thread
    $ux_1\ldots x_5v$.
    Let $y,z$ be the neighbors of $u$ distinct from $w,x_1$.
    Let $\vp$ be a strong odd $4$-coloring of $G_4-\{w,x_1,x_2,x_3,x_4\}$.
    Then $\vp(u),\vp(y),\vp(z)$ are pairwise distinct.
    Color both $w,x_1$ with a color in $\{\vp(y),\vp(z)\}\setminus\{\vp(v)\}$.
    Choose a color for $x_2$ not in $\vp(\{u,x_1\})$.
    By the moreover part of \Cref{lem:short-path-extensions}, % applied to $x_1x_2x_3x_4x_5v$, 
    we can extend $\vp$ to a strong odd $4$-coloring of $G_4$, a contradiction.
\end{proof}

\begin{lemma}\label{lem:star-extensions}
Let $u$ be a vertex of $G_4$.
If $d(u)=4$ and $u$ has a $1$-neighbor or $d(u)=3$, then  $u$ cannot be incident with three pairwise internally disjoint open $t_i$-threads where $(t_1,t_2,t_3)\in\{(1,3,4),(2,2,4),(3,3,3)\}$.
\end{lemma}
\begin{proof}
Suppose to the contrary that $u$ is incident with three pairwise internally disjoint open threads $ux^i_{1}\ldots x^i_{t_i}v^i$ for $i\in\{1,2,3\}$ for some  $(t_1,t_2,t_3)\in\{(1,3,4), (2,2,4), (3,3,3)\}$. 
Let  $X=\{u, w\}$ if $u$ is a $4$-vertex with a $1$-neighbor $w$, and let $X=\{u\}$ otherwise. 
Let $\vp$ be a strong odd $4$-coloring of $G_4-(X\cup\{x^i_j: 1\le i \le 3,\ 1\le j < t_i\})$.  
We first color $u$ with a color $\alpha$ where 
\begin{itemize}
    \item $\alpha\in\cset_4\setminus\vp(\{x^1_{t_1},v^1,v^3\})$ if $(t_1,t_2,t_3)=(1,3,4)$,
    \item $\alpha\in\cset_4\setminus\vp(\{x^1_{t_1},x^2_{t_2},v^3\})$ if $(t_1,t_2,t_3)=(2,2,4)$,
    \item $\alpha=\vp(x^3_{t_3})$ if $(t_1,t_2,t_3)=(3,3,3)$.
\end{itemize}
By \Cref{lem:short-path-extensions}, we can extend $\vp$ to $x^1_1, \ldots, x^1_{t_1}, x^2_1, \ldots, x^2_{t_2}$. 
We now show how to choose a color (that is not $\alpha$) for $x^3_1$ to guarantee that the strong odd condition holds at $u$. 
Let $p=\vp(x^1_1)$ and $q=\vp(x^2_1)$.
If $d(u)=3$, then color $x^3_1$ with $p$ if $p=q$ and with a color not in $\{\alpha,p,q\}$ otherwise.
If $d(u)=4$ (so $u$ has a $1$-neighbor $w$), then color $x^3_1$ with $p$, and color $w$ with $p$ if $p \neq q$ and with a color not in $\{\alpha,p\}$ otherwise.
Thus, the strong odd condition holds at $u$.
By the moreover part of \Cref{lem:short-path-extensions}, any color that is not $\alpha$ on $x^3_1$ can be extended to complete a strong odd $4$-coloring of $G_4$, a contradiction. 
\end{proof}

\begin{lemma}\label{lem:four-thread-bound}
Every $4$-vertex of $G_4$ is responsible for at most $18$ nearby $2$-vertices. 
\end{lemma}
\begin{proof}
Suppose to the contrary that $G_4$ has a $4$-vertex $u$ that is responsible for more than $18$ nearby $2$-vertices.
By \Cref{lem:common-closed-thread}, $u$ is not incident with a closed $4^+$-thread.
By \Cref{lem:five-path-extensions}(i), there is no open $6^+$-thread. 
Thus, $u$ must be incident with at least three distinct open $5$-threads $ux^i_1\ldots x^i_5v^i$ where $i\in\{1,2,3\}$. 
Let $w$ be the neighbor of $u$ distinct from $x^1_1, x^2_1, x^3_1$. Let $\vp$ be a strong odd $4$-coloring of $G_4-\{x^i_j: 1\le i \le 3,\ 1\le j \le 4\}$. 
Take distinct colors $\alpha, \beta$ not in $\vp(\{u,w\})$.
If $\alpha\notin\vp(\{v^1, v^2,v^3\})$, then extend $\vp$ by coloring $x^i_1$ with $\alpha$ for all $i$.
Otherwise, we may assume $\vp(v^{1})=\alpha$.
By the same argument, we may assume $\vp(v^{2})=\beta$.
Now, extend $\vp$ by coloring both $x^1_1, x^2_1$ with $\vp(w)$ and $x^3_1$ with a color in $\{\alpha, \beta\}\setminus\{\vp(v^3)\}$. 
Thus, the strong odd condition holds at $u$.
For each $i\in\{1,2,3\}$, choose a color for $x^i_2$ not in $\vp(\{u,x^i_1\})$.
By the moreover part of \Cref{lem:short-path-extensions},
% applied to each $x^i_1\ldots x^i_5v^i$, 
we can extend $\vp$ to a strong odd $4$-coloring of $G_4$, which is a contradiction.
\end{proof}

\begin{lemma}\label{lem:four-adjacent}
Let $u,v$ be adjacent vertices of $G_4$.
Suppose that each of $u,v$ is either a $4$-vertex with a $1$-neighbor or a $3$-vertex.
Then either $G_4$ contains a $5$-cycle or $u$ and $v$ cannot both be incident with two internally disjoint open $3^+$- and $4^+$-threads.
\end{lemma}
\begin{proof}
Assume that $G_4$ has no $5$-cycle.
Suppose to the contrary that each of $u, v$ is incident with two internally disjoint open $3^+$- and $4^+$-threads.
By \Cref{lem:common-closed-thread}, neither $u$ nor $v$ is incident with a closed thread.
For each $t\in\{3,4\}$, let $P^t=ux^t_1\ldots x^t_{p_t}x^t$ and $Q^t=vy^t_1\ldots y^t_{q_t}y^t$ be the corresponding maximal threads incident with $u$ and $v$, respectively, where $p_t,q_t\ge t$.
For convenience, let $x^t_{p_t+1}=x^t$ and $y^t_{q_t+1}=y^t$.
We call a thread among $P^3,P^4,Q^3,Q^4$ shared if its terminals are $u$ and $v$.

First, suppose that there is a shared thread $T=uz_1\ldots z_sv$.
Since $uv\in E(G_4)$ and $G_4$ has no $5$-cycle, $s\ne3$.
Thus $s\in\{4,5\}$ by \Cref{lem:five-path-extensions}(i).
Obtain $G'$ from $G_4$ by deleting all internal vertices of $T$ and the $1$-neighbors of $u,v$, if any.
Let $\vp$ be a strong odd $4$-coloring of $G'$.
Let $u'$ and $v'$ be the neighbors of $u$ and $v$ in $G'$ other than $v$ and $u$, respectively.
Since $d_{G'}(u)=d_{G'}(v)=2$, the colors on each of $\{u,v,u'\}$ and $\{u,v,v'\}$ are pairwise distinct.
If $d(u)=3$, then color $z_1$ with a color not in $\vp(\{u,v,u'\})$.
Otherwise, $d(u)=4$, so color both $z_1$ and the $1$-neighbor of $u$ with $\vp(u')$.
Similarly, color $z_s$ and the $1$-neighbor of $v$, if any.
Then the strong odd condition holds at $u,v$, and $\vp(z_1),\vp(z_s)\notin\{\vp(u),\vp(v)\}$.
Color $z_{s-1}$ with $\vp(u)$.
Since $\vp(u)\ne\vp(v)$, the strong odd condition holds at $z_s$.
Since $\vp(z_{s-1})=\vp(u)\ne\vp(z_s)$ and $s\in\{4,5\}$, by the moreover part of \Cref{lem:short-path-extensions}, 
% applied to $uz_1\ldots z_s$, 
we can extend $\vp$ to a strong odd $4$-coloring of $G_4$, a contradiction.

Now, suppose that there is no shared thread.
Obtain $G'$ from $G_4$ by deleting $u,v$, their $1$-neighbors, if any, and the vertices $x^t_1,\ldots,x^t_{t-1},y^t_1,\ldots,y^t_{t-1}$ for each $t\in\{3,4\}$.
Let $\vp$ be a strong odd $4$-coloring of $G'$.
Color $u$ with a color $\alpha\ne\vp(x^4_5)$.
By \Cref{lem:short-path-extensions}, extend $\vp$ to $x^3_2,x^3_1$, and let $\gamma=\vp(x^3_1)$.
Color $v$ with a color $\beta\notin\{\alpha,\gamma,\vp(y^4_5)\}$.
If $d(u)=3$, then color $x^4_1$ with a color not in $\{\alpha,\beta,\gamma\}$.
Otherwise, $d(u)=4$, so color both $x^4_1$ and the $1$-neighbor of $u$ with $\gamma$.
In either case, the strong odd condition holds at $u$.
By the moreover part of \Cref{lem:short-path-extensions}, 
% applied to $ux^4_1\ldots x^4_5$, 
we can extend $\vp$ to $x^4_2,x^4_3$.
By \Cref{lem:short-path-extensions}, extend $\vp$ to $y^3_2,y^3_1$.
If $d(v)=3$, then color $y^4_1$ with $\alpha$ when $\vp(y^3_1)=\alpha$, and with a color not in $\{\alpha,\beta,\vp(y^3_1)\}$ when $\vp(y^3_1)\neq\alpha$.
Otherwise, color $y^4_1$ with $\alpha$, and color the $1$-neighbor of $v$ with $\alpha$ when $\vp(y^3_1)\ne\alpha$, and with a color not in $\{\alpha,\beta\}$ when $\vp(y^3_1)=\alpha$.
In either case, the strong odd condition holds at $v$.
By the moreover part of \Cref{lem:short-path-extensions},
% applied to $vy^4_1\ldots y^4_5$,
we can extend $\vp$ to $y^4_2,y^4_3$.
Thus $\vp$ extends to a strong odd $4$-coloring of $G_4$, which is a contradiction.
\end{proof}

\subsection{Five colors}
The results in this subsection will be used for $c=5$, but we state and prove them for all $c\ge5$.

\begin{lemma}
\label{lem:c5-low-terminal}
Let $c\ge 5$.
Both terminals of every open $3$-thread of $G_c$ are $(c-1)^+$-vertices.
In particular, there is no $4^+$-thread in $G_c$.
\end{lemma}
\begin{proof}
Suppose to the contrary that there exists a $3$-thread $ux_1x_2x_3v$ with $d(u)\le c-2$.
Let $\vp$ be a strong odd $c$-coloring of $G_c-\{x_1,x_2\}$.
Since $|(N[u]\setminus\{x_1\})\cup\{x_3\}|\le d(u)+1\le c-1$, there is  a color $\alpha\notin\vp(N[u]\setminus\{x_1\})\cup\{\vp(x_3)\}$.
Since $c\geq 5$, there is a color $\beta\notin\vp(\{u,x_3,v\})\cup\{\alpha\}$.
Color $x_1, x_2$ with $\alpha,\beta$, respectively, to  obtain a strong odd $c$-coloring of $G_c$,  a contradiction.
\end{proof}

\begin{lemma}
\label{lem:two-low-threads}
Let $c\ge5$.  
Suppose two neighbors $x^1_1, x^2_1$ of a vertex $u$ of $G_c$ are on two open $2$-threads $ux^1_1x^1_2v^1$ and $ux^2_1x^2_2v^2$.
If $x^1_1$  is not on a closed thread, then at least one of $v^1, v^2$ is a $(c-2)^+$-vertex.
In particular, every vertex is incident with at most one open $3$-thread.
\end{lemma}
\begin{proof}
Suppose to the contrary that both $v^1$ and $v^2$ are $(c-3)^-$-vertices.
Let $\vp$ be a strong odd $c$-coloring of $G_c-\{x^1_1,x^1_2,x^2_1,x^2_2\}$. 
Let $X=\vp(N(u)\setminus\{x^1_1,x^2_1\})$.
If some $\alpha\in X$ is not in $\{\vp(v^1), \vp(v^2)\}$, then color $x^1_1,x^2_1$ with $\alpha$.
Otherwise $X\subseteq\vp(\{v^1, v^2\})$, so $X\cup\vp(\{u, v^1, v^2\})$ has size at most $3$.
We color $x^1_1, x^2_1$ with distinct colors $\alpha,\beta\not\in X\cup\vp(\{u,v^1,v^2\})$, respectively. 
In either case, the strong odd condition holds at $u$.
For each $i\in\{1,2\}$, sequentially color $x^i_2$ with a color not appearing on $\{u,x^i_1\}\cup(N[v^i]\setminus\{x^i_2\})$, which has size at most $d(v^i)+2\leq c-1$. 
This is a strong odd $c$-coloring of $G_c$, which is a contradiction.
\end{proof}

\begin{lemma}
\label{lem:two-two-one}
Let $c\ge 5$.
If $d\in\{3, \ldots, c-2\}$, then the neighbors of a $d$-vertex of $G_c$ are not on $d-1$ open $2^+$-threads and an  open $1^+$-thread.
\end{lemma}
\begin{proof}
Suppose to the contrary that the neighbors of a $d$-vertex $u$ are on $d-1$ open $2$-threads $ux_iy_iv_i$ for $i\in \{1,\ldots, d-1\}$ and an open $1$-thread $ux_dv_d$.
Let $\vp$ be a strong odd $c$-coloring of $G_c-\{u,x_1, \ldots, x_{d-1}\}$. 
Since $c>d+1$, we can extend $\vp$ to $u$ by coloring $u$ with a color not in $\vp(\{x_d,v_d,y_1,\ldots,y_{d-1}\})$.
For $i\in\{1,\ldots, d-1\}$, let $A_i=\cset_c\setminus\vp(\{u, y_i, v_i,x_d\})$, so $|A_i|\geq c-4\ge d-2$. 
If $A_1,\ldots,A_{d-1}$ have distinct representatives $(\alpha_1, \ldots, \alpha_{d-1})\in A_1\times\cdots\times A_{d-1}$, then color $x_i$ with  $\alpha_i$ for each $i\in\{1, \ldots, d-1\}$. 
Otherwise, Hall's theorem implies that all sets $A_i$ are the same set with $d-2$ elements.
This further implies $\vp(x_d)\in \cset_c\setminus\vp(\{u, y_i, v_i\})$ for every $i$.
Thus color $x_1,x_2$ with $\vp(x_d)$ and color $x_3,\ldots,x_{d-1}$ with distinct colors from $A_1$.
In either case, $\vp$ extends to a strong odd $c$-coloring of $G_c$, which is a contradiction.
\end{proof}

\section{Five colors for sparse graphs}\label{sec:five}

In this section, we prove \Cref{thm:main-five}. 
We begin by summarizing the useful reducible configurations for a strong odd 5-minimal graph from \Cref{sec:preliminaries}.

\begin{corollary}\label{cor:five-reductions}
A strong odd $5$-minimal graph satisfies the following:
\begin{enumerate}[label=\textup{(\roman*)}]
\item Every vertex has at most one $1$-neighbor, and the neighbor of a $1$-vertex has odd degree at least $5$.
\item The terminal of a closed thread is a $5^+$-vertex. Moreover, a $5$-vertex with a $1$-neighbor is not incident with a closed thread.
\item There is no $4^+$-thread. 
\item Both terminals of a $3$-thread are $4^+$-vertices. In particular, a $3$-vertex is not incident with a $3$-thread. 
\item Every vertex is incident with at most one open $3$-thread. 
\item The neighbors of a $3$-vertex are not on two open $2$-threads and an open $1$-thread. \end{enumerate}
\end{corollary}
\begin{proof}
(i) and (ii) follow from \Cref{lem:common-leaves} and \Cref{lem:common-closed-thread}, respectively.
(iii) and (iv) follow from \Cref{lem:c5-low-terminal}.
(v) and (vi) follow from \Cref{lem:two-low-threads} and \cref{lem:two-two-one}, respectively.
\end{proof}

Now, suppose to the contrary that there exists a graph $G$ with $\mad(G)<\frac{7}{3}$ that has no strong odd $5$-coloring.
Choose a strong odd $5$-minimal subgraph $G'$ of $G$.
Then $G'$ is connected and $\mad(G')<\frac{7}{3}$.
Moreover, $G'$ is not a cycle, since every cycle has a strong odd $5$-coloring. 

Let $X$ be the set of $1$-vertices of $G'$, and let $H=G'-X$.
By \Cref{cor:five-reductions}(i), each vertex in $X$ has a $5^+$-neighbor, so $H$ has minimum degree at least $2$  and the $2$-vertices of $H$ are exactly the $2$-vertices of $G'$.
Every $3^+$-vertex of $G'$ remains a $3^+$-vertex in $H$, so the maximal threads containing $2$-vertices are unchanged.
In particular, every $2$-vertex of $H$ lies on a maximal thread with $3^+$-terminals in $H$.

We obtain our final contradiction via discharging.
To each vertex $v$ of $H$, assign an initial charge of $d_H(v)$ and let $\mu(v)$ denote its final charge after applying the discharging rule.
Since the discharging procedure  preserves the total charge, we obtain
$\sum_{v\in V(H)}\mu(v)
=\sum_{v\in V(H)}d_H(v)
=2|E(H)|
<\frac{7}{3}|V(H)|
$.
We will show that the final charge of each vertex is at least $\frac{7}{3}$, which is a contradiction. 
The following is the single discharging rule: 

\begin{enumerate}[label=\textup{(R)},ref=(R),leftmargin=3em]
\item\label{rule:five-thread}
Every $3^+$-vertex $u$ sends charge $\frac{1}{6}$ to each nearby $2$-vertex $v$ 
\begin{enumerate}[label=\textup{(R\arabic*)},ref=(R\arabic*),leftmargin=3em]
    \item once if $v$ is on a maximal open thread incident with $u$. 
    \item twice if $v$ is on a closed thread incident with $u$. 
\end{enumerate}
\end{enumerate}

Let $u$ be a vertex of $H$. 
If $d_H(u)=2$, then $u$ is on a maximal thread $T$, so the  terminal(s) of $T$ sends charge a total of $2\cdot\frac{1}{6}$ to $u$.
Thus, $\mu(u)=2+\frac{2}{6}=\frac{7}{3}$.

Suppose $d_H(u)\ge3$, and let $u$ be responsible for $r$ nearby $2$-vertices. 
If $d_H(u)=3$, then $u$ is a $3$-vertex of $G'$.
By \Cref{cor:five-reductions}(ii), $u$ is not incident with a closed thread. 
Moreover, by \Cref{cor:five-reductions}(iv), (vi),  $r\le \max\{2+2+0,2+1+1\}=4$, so $\mu(u)\geq 3-4\cdot\frac16=\frac73$.

If $d_H(u)=4$, then $u$ is either a $5$-vertex of $G'$ with a $1$-neighbor or a $4$-vertex of $G'$.
By \Cref{cor:five-reductions}(ii), $u$ is not incident with a closed thread. 
Moreover, by \Cref{cor:five-reductions}(v), at most one thread incident with $u$ is an open $3$-thread.
Thus $r\le 3+2+2+2=9$, so $\mu(u)\geq 4-9\cdot\frac{1}{6}=\frac{5}{2}>\frac73$. 

If $d_H(u)\ge 5$, then $r\le3d_H(u)$ by \Cref{cor:five-reductions}(iii), so $\mu(u)=d_H(u)-\frac{r}{6}
\ge\frac{d_H(u)}{2}
>\frac{7}{3}$.

Hence, every vertex of $H$ has final charge at least $\frac{7}{3}$.

\section{Four colors for sparse graphs without induced $5$-cycles}
\label{sec:four}

In this section, we prove \Cref{thm:main-four}.
We first collect the necessary reducible configurations from \Cref{sec:preliminaries}.

\begin{corollary}\label{cor:four-reductions}
A strong odd $4$-minimal graph with no $5$-cycle satisfies the following:
\begin{enumerate}[label=\textup{(\roman*)}]
\item Every vertex has at most one $1$-neighbor, and the neighbor of a $1$-vertex has even degree at least $4$.
\item  Every closed thread has at most three internal vertices, and its terminal is a $4^+$-vertex.
Moreover, a $4$-vertex with a $1$-neighbor is not incident with a closed thread.
\item There is no $6^+$-thread.
\item If $u$ is a $4$-vertex with a $1$-neighbor or a $3$-vertex, then $u$ is responsible for at most eight nearby $2$-vertices, unless $u$ is a $3$-vertex and the neighbors of $u$ are on maximal $0$-, $4$-, $5$-threads.
\item Every $4$-vertex is responsible for at most $18$ nearby $2$-vertices.
\item Let $v$ be the $3^+$-neighbor of a $3$-vertex $u$ whose other two neighbors are on maximal $4$-, $5$-threads. 
If $v$ is a $4$-vertex with a $1$-neighbor or a $3$-vertex, then $v$ is responsible for at most seven nearby $2$-vertices.
\end{enumerate}
\end{corollary}
\begin{proof}
(i), (ii), (iii), (v) follow from \Cref{lem:common-leaves}, \Cref{lem:common-closed-thread}, \Cref{lem:five-path-extensions}(i), \Cref{lem:four-thread-bound}, respectively. 

(iv): Let $u$ be a $4$-vertex with a $1$-neighbor or a $3$-vertex.
By (ii), $u$ is not incident with a closed thread.
Excluding the possible $1$-neighbor of $u$, let $t_1,t_2,t_3$ denote the number of $2$-vertices on the three maximal threads incident with $u$.
We may assume $t_1\le t_2\le t_3$.
By (iii), $t_3\leq 5$, and by \Cref{lem:five-path-extensions}(ii), (iii), $t_2\leq 4$. 
If $t_1=0$, then $t_1+t_2+t_3\le0+4+5=9$, with equality only for $(t_1,t_2,t_3)=(0,4,5)$.
If $t_1\ge1$ and $t_3\le3$, then $t_1\le2$ by \Cref{lem:star-extensions} so  $t_1+t_2+t_3\le8$.
If $t_1\ge1$ and $t_3\ge4$, then $t_2\le2$ and $t_1=1$ by \Cref{lem:star-extensions}, so $t_1+t_2+t_3\le8$.
Finally, a $4$-vertex with a $1$-neighbor is not incident with a $5$-thread by \Cref{lem:five-path-extensions}(iii), so only a $3$-vertex can be responsible for nine nearby $2$-vertices. 

(vi): Suppose that $u,v$ satisfy the hypotheses.
By (ii), $v$ is not incident with a closed thread.
Excluding the possible $1$-neighbor of $v$, let $t_1,t_2,t_3$ denote the number of $2$-vertices on the three maximal threads incident with $v$.
We may assume $t_1=0$ and $t_2\le t_3$. 
By (iii), $t_3\leq 5$.
By \Cref{lem:four-adjacent}, either $t_2\le2$ or $t_3\le3$.
Thus $t_1+t_2+t_3\le7$.
\end{proof}

Let $G$ be a graph with $\mad(G)<\frac{11}{5}$ and no induced $5$-cycles.
Since a $5$-cycle together with a chord has average degree $\frac{12}{5}>\frac{11}{5}$, every $5$-cycle in $G$ must be induced, so $G$ has no $5$-cycles. Suppose to the contrary that $G$ has no strong odd $4$-coloring.
Choose a strong odd $4$-minimal subgraph $G'$ of $G$.
Then $G'$ is connected, has no $5$-cycles, and  $\mad(G')<\frac{11}{5}$.
Moreover, $G'$ is not a cycle, since every cycle other than a 5-cycle has a strong odd $4$-coloring.

Let $X$ be the set of $1$-vertices of $G'$, and let $H=G'-X$.
By \Cref{cor:four-reductions}(i), each vertex in $X$ has a $4^+$-neighbor, so $H$ has minimum degree at least $2$ and the $2$-vertices of $H$ are exactly the $2$-vertices of $G'$.
Every $3^+$-vertex of $G'$ remains a $3^+$-vertex in $H$, so the maximal threads containing $2$-vertices are unchanged.
In particular, every $2$-vertex of $H$ lies on a maximal thread with $3^+$-terminals in $H$.
Every $3$-vertex of $H$ is either a $4$-vertex of $G'$ with a $1$-neighbor or a $3$-vertex of $G'$.
Since the neighbor of a $1$-vertex has even degree by \Cref{cor:four-reductions}(i), every $4$-vertex of $H$ is a $4$-vertex of $G'$.

Let $B$ be the set of $3$-vertices of $H$ whose three neighbors are on maximal $0$-, $4$-, $5$-threads. 
By \Cref{cor:four-reductions}(iv), every vertex of $B$ is a $3$-vertex of $G'$.
Each vertex of $B$ has a unique $3^+$-neighbor in $H$, which is not in $B$, by \Cref{cor:four-reductions}(vi).

We obtain our final contradiction via discharging.
To each vertex $v$ of $H$, assign an initial charge of $d_H(v)$ and let $\mu(v)$ denote its final charge after applying the discharging rule.
Since the discharging procedure  preserves the total charge, we obtain
$\sum_{v\in V(H)}\mu(v)
=\sum_{v\in V(H)}d_H(v)
=2|E(H)|
<\frac{11}{5}|V(H)|
$.
We will show that the final charge of each vertex is at least $\frac{11}{5}$, which is a contradiction. 
We use the following   discharging rules: 

\begin{enumerate}[label=\textup{(R\arabic*)},ref=(R\arabic*),leftmargin=3em]
\item\label{rule:four-thread}
Every $3^+$-vertex $u$ sends charge $\frac1{10}$ to each nearby $2$-vertex $v$ 
\begin{enumerate}[label=\textup{(R1\alph*)},ref=(R\alph*),leftmargin=3em]
    \item once if $v$ is on a maximal open thread incident with $u$. 
    \item twice if $v$ is on a closed thread incident with $u$. 
\end{enumerate}

\item\label{rule:four-exceptional}
Every $3^+$-vertex sends charge $\frac1{10}$ to each neighbor in $B$.
\end{enumerate}

Let $u$ be a vertex of $H$.
If $d_H(u)=2$, then $u$ is on a maximal thread $T$, so the terminal(s) of $T$ sends charge a total of $2\cdot\frac1{10}$ to $u$. 
Thus, $\mu(u)=2+\frac2{10}=\frac{11}{5}$. 
If $u\in B$, then $u$ sends charge $9\cdot\frac1{10}$ by \ref{rule:four-thread} and  receives $\frac1{10}$ from its unique $3^+$-neighbor by \ref{rule:four-exceptional}.
Note that $u$ has no neighbor in $B$. 
Hence $\mu(u)=3-\frac9{10}+\frac1{10}=\frac{11}{5}$.

Suppose that $d_H(u)\geq 3$ and $u\notin B$.
Let $u$ be responsible for $r$ nearby $2$-vertices, and let $b$ be the number of neighbors of $u$ in $B$.
Note that $u$ sends charge $(r+b)\cdot\frac1{10}$.
If $d_H(u)=3$, then $u$ is either a $4$-vertex with a $1$-neighbor in $G'$ or a $3$-vertex in $G'$. 
Since $u\notin B$, we know $r\le8$ by \Cref{cor:four-reductions}(iv).
If $b=1$, then $r\le7$ by \Cref{cor:four-reductions}(vi).
If $b\ge2$, then since each neighbor in $B$ is a $3^+$-vertex, $r+b\le5(3-b)+b=15-4b\le7$ by \Cref{cor:four-reductions}(iii).
Thus, $r+b\le8$ in every case, so $\mu(u)\ge3-\frac8{10}=\frac{11}{5}$.

If $d_H(u)=4$, then $r\le18$ by \Cref{cor:four-reductions}(v).
If $b\ge1$, then $r+b\le5(4-b)+b=20-4b\le16$ by \Cref{cor:four-reductions}(iii).
Thus, $r+b\le18$ in every case, so $\mu(u)\ge4-\frac{18}{10}=\frac{11}{5}$.

If $d_H(u)\ge5$, then \Cref{cor:four-reductions}(iii) gives
$r+b\le5(d_H(u)-b)+b=5d_H(u)-4b\le5d_H(u)$,
so $\mu(u)\ge d_H(u)-\frac{d_H(u)}2\ge\frac52>\frac{11}{5}$.

Hence, every vertex of $H$ has final charge at least $\frac{11}{5}$.

\section{Asymptotically tight bound}\label{sec:below-four}

In this section, we prove \Cref{thm:c-bound}.  
An \emph{orientation} of a graph gives a direction to each edge. 
We use a fact attributed to Hakimi~\cite{Hakimi1965}.
Let $\Delta^+(D)$ denote the maximum outdegree of an orientation $D$. 

\begin{theorem}[\!\!\cite{Hakimi1965}]\label{thm:Hakimi-orientation}
Let $G$ be a graph and let $k$ be a nonnegative integer.
Then $G$ has an orientation $D$ with $\Delta^+(D)\le k$
if and only if $\mad(G)\le 2k$.
\end{theorem}

An edge-coloring of a loopless multigraph is \emph{odd} if, at each vertex, every color appears $0$ or an odd number of times on its incident edges. 
We use the following theorem of Lu\v{z}ar, Petru\v{s}evski, and \v{S}krekovski~\cite{LuzarPetrusevskiSkrekovski2015}.

\begin{theorem}[\!\!\cite{LuzarPetrusevskiSkrekovski2015}]\label{thm:odd-six}
Every loopless multigraph has an odd $6$-edge-coloring.
\end{theorem}

We establish the below statement, from which \Cref{thm:c-bound} follows. 

\begin{proposition}\label{prop:epsilon-bound}
For  $\eps\in(0, 4)$, let $r=\lceil 10-\eps+\frac8\eps\rceil$. 
If $G$ is a graph with $\mad(G)<4-\eps$, then $\socol(G)\le6r$.
\end{proposition}
\begin{proof}
Let $G$ be a graph with $\mad(G)<4-\eps$.
Since $\mad(G)<4$, $G$ has an orientation $D$ with $\Delta^+(D)\leq 2$ by \Cref{thm:Hakimi-orientation}.
For each vertex $y$, let 
$I_y=\{x\in N(y):x\to y\text{ is an arc of }D\}$ and $O_y=\{x\in N(y):y\to x\text{ is an arc of }D\}$, so $I_y$ and $O_y$ form a partition of $N(y)$.
Since $\Delta^+(D)\leq 2$, we know $|O_y|\le 2$ and every vertex of $G$ belongs to $I_y$ for at most two vertices $y$.

We define two auxiliary simple graphs $H_1$ and $H_2$ on $V(G)$. 
Let $uv$ be an edge of $H_1$, if there is a vertex $y$, called a $1$-witness of $uv$, such that either $(u, v)\in I_y\times O_y$ or $(u, v)\in O_y\times I_y$; namely, a path $u\to y\to v$ or $v\to y\to u$ exists.
We claim $\mad(H_1)\le 8$, by showing that $H_1$ has an orientation $D_1$ with $\Delta^+(D_1)\leq 4$ and applying \Cref{thm:Hakimi-orientation}. 
Indeed, for each edge $uv$ of $H_1$, choose a $1$-witness $y$ and orient $uv$ as follows: $u\to v$ if $(u, v)\in I_y\times O_y$, and $v\to u$, otherwise. 
For a vertex $u$, $d^+_{D_1}(u)\leq\sum_{y:u\in I_y}|O_y|\leq 2\cdot2=4$, so $\Delta^+(D_1)\leq 4$. 

Next, let $uv$ be an edge of $H_2$  whenever $O_y=\{u,v\}$ for some vertex $y$, called a $2$-witness of $uv$.  
Since $\Delta^+(D)\leq 2$, each vertex of $G$ determines at most one edge of $H_2$, so distinct edges of $H_2$ have distinct $2$-witnesses. 
% the vertices $y_{e}$ are distinct.
Let $J$ be a nonempty subgraph of $H_2$ with $n$ vertices and $m$ edges, so its average degree is $d=\frac{2m}n$. 
For each edge $e=uv$ of $J$, choose a $2$-witness $y_e$ and let $W=V(J)\cup\{y_e:e\in E(J)\}$.
Since the $2$-witnesses are distinct, $|W|\le n+m$ and $|E(G[W])|\geq2m$, so
$\frac{4d}{d+2}
=\frac{4m}{n+m}
\le\frac{2|E(G[W])|}{|W|}
\le\mad(G)<4-\eps$.
Hence $d<\frac8\eps-2$, and therefore $\mad(H_2)<\frac8\eps-2$.

Let $H=G\cup H_1\cup H_2$. 
Then $\mad(H)\le \mad(G)+\mad(H_1)+\mad(H_2)
<(4-\eps)+8+\left(\frac{8}{\eps}-2\right)
=10-\eps+\frac{8}{\eps}\le r$.
Thus $H$ is $(r-1)$-degenerate, so it has a proper $r$-coloring $\nu:V(G)\to\{1,\ldots, r\}$. For each $j\in\{1,\ldots, r\}$, let $L_j=\nu^{-1}(j)$.
Two distinct vertices $u,v$ cannot receive the same color if $u,v\in O_y$ since $uv\in E(H_2)$, and if $(u, v)\in I_y\times O_y$ since $uv\in E(H_1)$. 
Therefore, for every vertex $y$ and $j\in\{1, \ldots, r\}$, either $L_j\cap O_y=\emptyset$, or $|L_j\cap O_y|=1$ and $L_j\cap I_y=\emptyset$.\hfill $(\ast)$

For each $j\in\{1,\ldots, r\}$, we construct a loopless multigraph $M_j$, whose edges correspond to the vertices of $L_j$. 
For every vertex $x$, since $|O_x|\leq 2$, define $e_x$ as follows:
\begin{itemize}
    \item If $O_x=\{y,y'\}$, then let $e_x=yy'$. 
    \item If $O_x=\{y\}$, then introduce a new vertex $z_x$ and let $e_x=yz_x$. 
    \item If $O_x=\emptyset$, then introduce two new vertices $z_x, z'_x$ and let $e_x=z_xz'_x$.
\end{itemize}
Let $M_j$ be the multigraph with vertex set consisting of $V(G)$ together with all newly introduced vertices, and with edge multiset $E(M_j)=\{e_x:x\in L_j\}$.
In particular, $x\mapsto e_x$ is a bijection between $L_j$ and $E(M_j)$.
By construction, for every $x\in L_j$ and $y\in V(G)$,
$e_x$ is incident with $y$ in $M_j$ if and only if $x\in I_y$.
By \Cref{thm:odd-six}, the multigraph $M_j$ has an odd $6$-edge-coloring
$\psi_j:E(M_j)\to\{1,\ldots,6\}$.

Finally, we define a $6r$-coloring  $\vp:V(G)\to\{1, \ldots, r\}\times\{1, \ldots, 6\}$ of $G$ where  $\vp(x)=(\nu(x),\psi_{\nu(x)}(e_x))$.
Since $\nu$ is a proper coloring of $H$ and $G$ is a subgraph of $H$, we know $\vp$ is a proper coloring. 
We now verify the strong odd condition.
Fix $v\in V(G)$ and consider a color $(j,t)$.
If $L_j\cap O_v\ne\emptyset$, then by $(\ast)$, $L_j\cap N(v)$ consists of a single vertex, so $(j,t)$ appears at most once in $N(v)$.
Otherwise, $L_j\cap O_v=\emptyset$, so $L_j\cap N(v)=L_j\cap I_v$.
By the construction of $M_j$, we have
$|\{x\in N(v):\vp(x)=(j,t)\}|
=|\{x\in L_j\cap I_v:\psi_j(e_x)=t\}|
=|\{e\in E(M_j):v\in e, \psi_j(e)=t\}|$, which is $0$ or odd since $\psi_j$ is an odd edge-coloring.
Thus $\vp$ is a strong odd coloring of $G$ with at most $6r$ colors.
\end{proof}

To prove \Cref{thm:c-bound}, let $\eps=\frac{96}{c-65+\sqrt{(c-65)^2+1152}}$, so $\eps\in(0, 4)$ for $c\ge54$;
equivalently, $\eps$ is the positive solution of
$10-\eps+\frac8\eps=\frac{c-5}{6}$.
Hence, by \Cref{prop:epsilon-bound}, if $\mad(G)<4-\eps$, then 
$\socol(G)\le
6\left\lceil\frac{c-5}{6}\right\rceil 
=6\left\lfloor\frac{c}{6}\right\rfloor\le c$. 

\section{Future directions}\label{sec:concluding}

We conclude with a potential future research direction.
For $c\geq 5$, let $m_c=\inf\{\mad(G): \socol(G)\geq c+1 \}$.
In other words, $m_c$ is the threshold on the maximum average degree guaranteeing a strong odd $c$-coloring. 
\Cref{thm:main-five} and the construction of $Q_5$ imply $m_5=\frac73$.
It would be interesting to determine $m_c$ for all $c\geq 6$. 

\begin{problem}
    Determine $m_c$ for all $c\geq 6$. 
\end{problem}

In particular, we conjecture $m_6=\frac{18}7$.

\begin{conjecture}\label{conj:six}
For a graph $G$, if $\mad(G)<\frac{18}{7}$, then $\socol(G)\le 6$.
\end{conjecture}

The bound on the maximum average degree would be best possible.
Let $Q_6$ be the graph obtained from two $4$-cycles sharing exactly one vertex and adding an edge between the two vertices farthest away from the shared vertex. 
Then $\mad(Q_6)=\frac{18}7$ and $\socol(Q_6)=7$ as all seven vertices must receive distinct colors in every strong odd coloring of $Q_6$.

We remark that \Cref{conj:six} holds for every graph $G$ with $\Delta(G)\leq 3$, as Havet~\cite{Havet2009} proved that $G^2$ is $6$-choosable for every graph $G$ with $\Delta(G)\leq 3$ and $\mad(G)<\frac{18}7$.

\bigskip
\noindent\textbf{Declaration of LLM use.}
\par\smallskip
\noindent
The authors acknowledge the use of OpenAI's ChatGPT, which suggested a connection to odd edge-coloring that was helpful in developing the proof of \Cref{thm:c-bound} and provided editorial assistance.
The authors independently developed and verified all mathematical arguments, revised the exposition, and take full responsibility for the contents of the article.

% \bibliographystyle{amsplain}
% \bibliography{references}

\end{document}